\documentclass{amsart}
\usepackage{amsmath,amssymb}
\usepackage{graphicx,subfigure}
\usepackage{mathtools,slashed}
\newtheorem{theorem}{Theorem}[section]
\newtheorem{proposition}[theorem]{Proposition}
\newtheorem{corollary}[theorem]{Corollary}

\theoremstyle{definition}

\theoremstyle{remark}

\numberwithin{equation}{section}

\newcommand{\al}{\alpha}

\newcommand{\de}{\delta}

\newcommand{\vphi}{\varphi}

\newcommand{\la}{\lambda}
\newcommand{\om}{\omega}
\newcommand{\si}{\sigma}

\newcommand{\De}{\Delta}

\newcommand{\Si}{\Sigma}

\def\hg{\widehat g}

\newcommand{\tpsi}{\widetilde{\psi}}
\newcommand{\tphi}{\widetilde{\phi}}
\newcommand{\tvphi}{\widetilde{\vphi}}

\def\CC{\mathbb{C}}

\def\RR{\mathbb{R}}
\def\BB{\mathbb{B}}
\def\ZZ{\mathbb{Z}}

\def\TT{\mathbb{T}}
\def\MM{{\mathbb{M}^n}}

\renewcommand\SS{\mathbb{S}}
\newcommand\minus\backslash

\newcommand\lan\langle
\newcommand\ran\rangle

\newcommand{\e}{{e}}

\DeclareMathOperator\dist{dist}

\renewcommand\leq\leqslant
\renewcommand\geq\geqslant
\newlength{\intwidth}

\begin{document}

\title[Nodal sets of high-energy eigenfunctions]{High-energy eigenfunctions of the Laplacian on the torus and the sphere with nodal sets of complicated topology}

\author{A. Enciso}
\author{D. Peralta-Salas}
\address{Instituto de Ciencias Matem\'aticas, Consejo Superior de
  Investigaciones Cient\'\i ficas, 28049 Madrid, Spain}
\email{aenciso@icmat.es, dperalta@icmat.es}

\author{F. Torres de Lizaur}
\address{Max-Planck-Institut f\"{u}r Mathematik, Vivatsgasse 7, 53111 Bonn, Germany}
\email{ftorresdelizaur@mpim-bonn.mpg.de}

\begin{abstract}
Let $\Sigma$ be an oriented compact hypersurface in the round sphere $\SS^n$ or in the flat
torus $\TT^n$, $n\geq 3$. In the case of the torus, $\Sigma$ is further assumed to
be contained in a contractible subset of $\TT^n$. We show that for any
sufficiently large enough odd integer $N$ there exists an eigenfunctions $\psi$ of the Laplacian on
$\SS^n$ or $\TT^n$ satisfying $\Delta \psi=-\la \psi$ (with $\lambda=N(N+n-1)$ or $N^2$ on $\SS^n$ or $\TT^n$, respectively), and with a connected component of the nodal set of $\psi$ given by~$\Sigma$, up to an ambient diffeomorphism.
\end{abstract}

\maketitle

\section{Introduction}

Let $M$ be a closed manifold of dimension
$n\geq 3$ endowed with a smooth Riemannian metric $g$. The
Laplace eigenfunctions of $M$ satisfy the equation
\[
\De u_k=-\la_k u_k\,,
\]
where $0=\la_0<\la_1\leq\la_2\leq\dots$
are the eigenvalues of the Laplacian. The zero set $u_k^{-1}(0)$ is called the nodal set
of the eigenfunction.

The study of the nodal sets of the
eigenfunctions of the Laplacian in a compact Riemannian manifold
is a classical topic in geometric analysis with a number of
important open problems~\cite{Ya82,Ya93}. When the Riemannian metric is not fixed, the nodal set is quite flexible. Indeed, it has been recently shown that~\cite{EP16}, given
a separating hypersurface $\Si$ in $M$, there is a metric~$g$ on the
manifold for which the nodal set $u_1^{-1}(0)$ of the first
eigenfunction is precisely~$\Si$. This result has been extended to the class of metrics conformal to a metric $g_0$ prescribed a priori~\cite{EPS17}, and to higher codimension submanifolds arising as the joint nodal set of several eigenfunctions corresponding to a degenerate eigenvalue~\cite{EHP16}.

For a fixed Riemannian metric, the problem is much more rigid than when one can freely choose a metric
adapted to the geometry of the hypersurface that one aims to
recover from the nodal set of the eigenfunctions. In this case, the techniques developed in~\cite{EP16,EHP16,EPS17} do not work. Nevertheless, since the Hausdorff measure of the nodal sets of the eigenfunctions grows as the eigenvalue tends to infinity~\cite{Lo1,Lo2}, one expects that the nodal set may become topologically complicated for high-energy eigenfunctions.

Our goal in this paper is to establish the existence of high-energy eigenfunctions of the Laplacian on the round sphere $\SS^n$ and the flat torus $\TT^n$ with nodal sets diffeomorphic to a given submanifold. All along this paper,~$\SS^n$ denotes the unit sphere in $\RR^{n+1}$ and $\TT^n$ is the standard flat n-torus, $(\RR/2\pi\ZZ)^n$.

More precisely, our main theorem shows that for a sequence of high enough eigenvalues, there exist $m$ eigenfunctions of the Laplacian on $\SS^n$ or $\TT^n$ with a joint nodal set diffeomorphic to a given codimension $m$ submanifold $\Sigma$. For the construction we need to assume that the normal bundle of $\Sigma$ is trivial. This means that a small tubular neighborhood of the submanifold~$\Sigma$ must be
diffeomorphic to $\Sigma\times\RR^m$. In the statement, structural stability means that any small enough perturbation of the corresponding
eigenfunction (in the $C^k$ norm with $k\geq1$) still has a union of connected
components of the nodal set that is diffeomorphic to the submanifold $\Sigma$ under consideration. Throughout, diffeomorphisms are of class $C^\infty$ and submanifolds are $C^\infty$ and without boundary.

\begin{theorem}\label{T.high}
Let $\Sigma$ be a finite union of (disjoint, possibly knotted or linked) codimension $m\geq 1$ compact
submanifolds of $\SS^n$ or $\TT^n$, $n\geq 3$, with trivial normal bundle. In the case of the torus, we further assume that $\Sigma$ is contained in a contractible subset. If $m=1$, we also assume that $\Sigma$ is connected. Then for any large
enough odd integer $N$ there
are $m$ eigenfunctions $\psi_1,\dots,\psi_m$ of the Laplacian with
eigenvalue $\la=N(N+n-2)$ (in $\SS^n$) or $\lambda=N^2$ (in $\TT^n$), and a diffeomorphism $\Phi$ such that
$\Phi(\Sigma)$ is the union of connected components of the joint nodal set
$\psi_1^{-1}(0)\cap\cdots\cap\psi_m^{-1}(0)$. Furthermore, $\Phi(\Sigma)$ is structurally stable.
\end{theorem}

An important observation is that the proof of this theorem yields a
reasonably complete understanding of the behavior of the
diffeomorphism~$\Phi$, which is, in particular, connected with the identity. Oversimplifying a little, the
effect of $\Phi$ is to uniformly rescale a contractible subset of the
manifold that contains~$\Sigma$ to have a diameter of order
$1/N$. In particular, the control that we have over the
diffeomorphism~$\Phi$ allows us to prove an analog of this result for
quotients of the sphere by finite groups of isometries (lens spaces). Notice that $\Phi(\Sigma)$ is not guaranteed to contain all the components of the nodal set of the eigenfunction.

The proof of the main theorem involves an interplay between rigid and
flexible properties of high-energy eigenfunctions of the Laplacian. Indeed, rigidity
appears because high-energy eigenfunctions in any Riemannian $n$-manifold behave,
locally in sets of diameter $1/\sqrt\lambda$, as monochromatic waves in $\RR^n$
do in balls of diameter~1. We recall that a monochromatic wave is any solution to the Helmholtz equation $\Delta \phi+\phi=0$. The catch here is
that, in general, one cannot check whether a given monochromatic wave in
$\RR^n$ actually corresponds to a high-energy eigenfunction on the
compact manifold.

To prove the converse implication, what we call an inverse localization theorem (see Sections~\ref{IL_sphere} and~\ref{IL_torus}), it is key to exploit some flexibility that arises in the problem as a consequence of the fact that large eigenvalues of the Laplacian in the torus or in the sphere have increasingly high multiplicities (for this reason the proof does not work in a general
Riemannian manifold). The inverse localization is a powerful tool to ensure that any monochromatic wave in a compact set of $\mathbb R^n$ can be reproduced in a small ball of the manifold by a high-energy eigenfunction. This allows us to transfer any structurally stable nodal set that can be realized in Euclidean space to high-energy eigenfunctions on $\mathbb S^n$ and $\mathbb T^n$. The inverse localization was first introduced in~\cite{EPT} to construct high-energy Beltrami fields on the torus and the sphere with topologically complicated vortex structures, and was also exploited in~\cite{EPH1,EPH2} to solve a problem of M. Berry~\cite{Berry} on knotted nodal lines of high-energy eigenfunctions of the harmonic oscillator and the hydrogen atom, and in~\cite{T18} to analyze the nodal sets of the eigenfunctions of the Dirac operator.

One should notice that the techniques introduced in~\cite{EP}
to prove the existence of solutions to second-order elliptic PDEs in $\RR^n$ (including the monochromatic waves) with a prescribed
nodal set $\Sigma$ do not work for
compact manifolds. The reason is that the proof is based on the
construction of a local solution in a neighborhood of $\Sigma$,
which is then approximated by a global solution in $\RR^n$ using
a Runge-type global approximation theorem. For compact manifolds the
complement of the set $\Sigma$ is precompact, so we cannot apply the
global approximation theorem obtained in~\cite{EP}. In fact,
as is well known, this is not just a technical issue, but a
fundamental obstruction in any approximation theorem of this
sort. This invalidates the whole strategy followed
in~\cite{EP} and makes it apparent that new tools are needed
to prove the existence of Laplace eigenfunctions with geometrically complex
nodal sets in compact manifolds.

We finish this introduction with two corollaries. It is known that an oriented codimension one or two submanifold in $\SS^n$ or $\TT^n$ has
trivial normal bundle~\cite{Ma59}, therefore the main theorem implies the following:

\begin{corollary}
Let $\Sigma$ be an oriented, compact, connected hypersurface in $\SS^n$ or $\TT^n$, $n\geq 3$. In the case of the torus, we further assume that $\Sigma$ is contained in a contractible subset. Then for any large
enough odd integer $N$ there
is an eigenfunction $\psi$ of the Laplacian with
eigenvalue $\la=N(N+n-2)$ (in $\SS^n$) or $\lambda=N^2$ (in $\TT^n$), and a diffeomorphism $\Phi$ such that
$\Phi(\Sigma)$ is a structurally stable connected component of the nodal set
$\psi^{-1}(0)$.
\end{corollary}

\begin{corollary}
Let $\Sigma$ be a finite union of (disjoint, possibly knotted or linked) codimension two compact
submanifolds in $\SS^n$ or $\TT^n$, $n\geq 3$. In the case of the torus, we further assume that $\Sigma$ is contained in a contractible subset. Then for any large enough odd integer $N$ there
is a complex-valued eigenfunction $\psi$ of the Laplacian with
eigenvalue $\la=N(N+n-2)$ (in $\SS^n$) or $\lambda=N^2$ (in $\TT^n$), and a diffeomorphism $\Phi$ such that
$\Phi(\Sigma)$ is a union of structurally stable connected components of the nodal set
$\psi^{-1}(0)$.
\end{corollary}

The paper is organized as follows. In Sections~\ref{IL_sphere} and~\ref{IL_torus} we prove an inverse localization theorem for the eigenfunctions of the Laplacian on $\SS^n$ and $\TT^n$, respectively. Theorem~\ref{T.high} is then proved in Section~\ref{P.main}. Finally, in Section~\ref{multiple}, we prove a refinement of the inverse localization Theorem on $\SS^n$ that allows us to approximate several given monochromatic waves by a single eigenfunction of the Laplacian in different small regions of $\SS^n$.

\section{An inverse localization theorem on the sphere}\label{IL_sphere}

In this section we prove an inverse localization theorem for eigenfunctions of the Laplacian on $\SS^n$ for $n\geq 2$. We recall that the eigenvalues of the Laplacian on the $n$-sphere are of the form $N(N+n-1)$, where $N$ is a nonnegative integer, and the corresponding multiplicity is given by
\[
d(N, n):=\binom{N+n-1}{N} \frac{2N+n-1}{N+n-1}\,.
\]

For the precise statement of the theorem, let us fix an arbitrary point $p_0\in \SS^n$ and take a patch of normal geodesic coordinates
$\Psi:\BB\to B$ centered at $p_0$. Here and in what follows, $B_\rho$ (resp.~$\BB_\rho$) denotes the
ball in~$\RR^n$ (resp.\ the geodesic ball
in~$\SS^n$) centered at the origin (resp.\ at $p_0$) and of radius $\rho$, and we shall drop the subscript when $\rho=1$. For the ease of notation, we will use the $\RR^m$-valued functions $\phi:=(\phi_1,\cdots,\phi_m)$ and $\psi:=(\psi_1,\cdots,\psi_m)$, and the action of the Laplacian on such functions is understood componentwise.

\begin{theorem}\label{T.spharm1}
Let $\phi$ be an $\RR^m$-valued monochromatic wave in $\RR^n$, satisfying
$\Delta \phi+\phi=0$. Fix a positive integer $r$ and a positive constant $\de'$. For any large enough integer $N$, there is an $\RR^m$-valued eigenfunction $\psi$ of the Laplacian on
$\SS^n$ with eigenvalue $N(N+n-1)$ such that
\begin{equation*}
\bigg\|\phi-\psi\circ \Psi^{-1}\Big(\frac\cdot N\Big)\bigg\|_{C^{r}(B)}\leq\de'\,.
\end{equation*}

\end{theorem}

To prove Theorem~\ref{T.spharm1}, we will proceed in two successive approximation steps. First, we will approximate the function $\phi$ in $B$ by an $\RR^m$-valued function $\vphi$ that can be written as a finite sum of terms of the form
\[
\frac{c_{j}}{|x-x_j|^{\frac{n}{2}-1}}J_{\frac{n}{2}-1}(|x-x_j|)
\]
with $c_j\in \RR^m$ and $x_j\in\RR^n$, $j=1,...,N'$, for $N'$ large enough (Proposition~\ref{P.Bessel1} below). Notice that any function of this form is a monochromatic wave. In the second step, we show that there is a collection of $m$ eigenfunctions $(\psi_1,...,\psi_m)=:\psi$ in $\SS^n$ with eigenvalue $N(N+n-1)$ such that, when considered in a ball of radius $N^{-1}$, they approximate $\vphi:=(\vphi_{1},...,\vphi_{m})$ in the unit ball, provided that $N$ is large enough.
\begin{proposition}\label{P.Bessel1}
Given any $\de>0$, there is a constant $R>0$ and finitely many
constant vectors $\{c_j\}_{j=1}^{N'}\subset \RR^m$ and points $\{x_j\}_{j=1}^{N'}\subset B_R$ such that the
function
\[
\vphi:=\sum_{j=1}^{N'} \frac{c_j}{|x-x_j|^{\frac{n}{2}-1}} J_{\frac{n}{2}-1}(|x-x_j|)
\]
approximates the function~$\phi$ in the unit ball as
\[
\|\phi-\vphi\|_{C^{r}(B)}<\de\,.
\]
\end{proposition}

\begin{proof}
It is more convenient to work with complex-valued functions, so we set $\tphi:=\phi+i \phi$. First, we notice that, since $\tphi$ is also a solution of the Helmholtz equation, it can be written in the
ball $B_2$ as an expansion
\begin{equation}\label{fbser}
\tphi=\sum_{l=0}^\infty\sum_{k=1}^{d(l,n-1)} b_{lk}\,j_{l}(r)\, Y_{lk}(\om),
\end{equation}
where $r:=|x|\in\RR^+$ and $\om:=x/r\in\SS^{n-1}$ are spherical coordinates in $\RR^{n}$, $Y_{lk}$ is a basis of spherical harmonics of eigenvalue $l(l+n-2)$, $j_l$ are $n$-dimensional hyperspherical Bessel functions and $b_{lk}\in\CC^{m}$ are constant coefficients.

The series in \eqref{fbser} is convergent in the $L^2$ sense, so for any $\de'>0$, we can truncate the sum at some integer $L$
\begin{equation}\label{Lrvb}
\phi_1:=\sum_{l=0}^{L}\sum_{k=1}^{d(l,n-1)} b_{lk}\, j_{l}(r)\, Y_{lk}(\om)
\end{equation}
so that it approximates $\tphi$ as
\begin{equation}\label{L2b}
\|\phi_1-\tphi\|_{L^2(B_2)}<\de'\,.
\end{equation}

The $\CC^m$-valued function $\phi_1$ decays as $|\phi_{1}(x)|\leq C/|x|^{\frac{n-1}{2}}$ for large enough $|x|$ (because of the decay properties of the spherical Bessel functions). Hence, Herglotz's
theorem (see e.g.~\cite[Theorem 7.1.27]{Hormander}) ensures that we can write
\begin{equation}
\phi_1(x)=\int_{\SS^{n-1}}f_1(\xi)\, e^{ix\cdot\xi}\, d\si(\xi)\,,
\end{equation}
where $d\si$ is the area measure on $\SS^{n-1}:=\{\xi\in\RR^n:|\xi|=1\}$ and $f_1$ is a $\CC^{m}$-valued function in $L^2(\SS^{n-1})$.

We now choose a smooth $\CC^{m}$-valued function $f_2$ approximating $f_1$ as
\[
\|f_1-f_2\|_{L^2(\SS^{n-1})}<\de'\,,
\]
which is always possible since smooth functions are dense in $L^2(\SS^{n-1})$. The function defined as the inverse Fourier transform of $f_2$,
\begin{equation}
\phi_2(x):=\int_{\SS^{n-1}}f_2(\xi)\,e^{ix\cdot\xi}\, d\si(\xi)\,,
\end{equation}
approximates $\phi_1$ uniformly: by the
Cauchy--Schwarz inequality, we get
\begin{align}
|\phi_2(x)-\phi_1(x)|=\bigg|\int_{\SS^{n-1}}(f_2(\xi)-f_1(\xi))\,e^{ix\cdot\xi}\,
  d\si(\xi)\bigg|\leq C\|f_2-f_1\|_{L^2(\SS^{n-1})}<C\de'\, \label{eqv2v1}
\end{align}
for any $x\in\RR^n$.

Our next objective is to approximate the function $f_2$ by a trigonometric polynomial: for any given $\de'$, we will find a constant $R>0$,
finitely many points
$\{x_j\}_{j=1}^{N'}\subset B_R$ and constants $\{c_j\}_{j=1}^{N'}\subset\CC^{m}$ such that the smooth function in $\RR^n$
$$
f(\xi):=\frac{1}{(2\pi)^{\frac{n}{2}}}\sum_{j=1}^{N'} c_j \, e^{-ix_j\cdot \xi}\,,
$$
when restricted to the unit sphere, approximates $f_2$ in the $C^0$ norm,
\begin{equation}\label{eqqma}
\|f-f_2\|_{C^0(\SS^{n-1})}<\de'\,.
\end{equation}

In order to do so, we begin by extending $f_2$ to a smooth function $g:\RR^n\to \CC^{m}$ with compact support,
$$
g(\xi):=\chi(|\xi|)\, f_2\bigg(\frac{\xi}{|\xi|}\bigg)\,,
$$
where $\chi(s)$ is a real-valued smooth bump function, being $1$ when, for example,
$|s-1|<\frac14$, and vanishing for $|s-1|>\frac12$. The Fourier
transform $\hg$ of $g$ is Schwartz, so it is easy to see that, outside some ball $B_R$, the $L^1$~norm of $\hg$ is very small,
\[
\int_{\RR^n\backslash B_R}|\hg(x)|\, dx<\de'\,,
\]
and therefore we get a very good approximation of $g$ by just considering its Fourier representation with frequencies within the ball $B_R$, that is,
\begin{equation}\label{Rlarge}
\sup_{\xi\in\RR^n}\bigg|g(\xi)-\int_{B_R}\hg(x)\, e^{-ix\cdot\xi}\, dx\bigg|<\de'/2\,.
\end{equation}

Next, let us show that we can approximate the integral
$$
\int_{B_R}\hg(x)\,e^{-ix\cdot\xi}\, dx
$$
by the sum
\begin{equation}\label{fcn}
f(\xi):=\frac{1}{(2\pi)^{\frac{n}{2}}}\sum_{j=1}^{N'} c_j \, e^{-ix_j\cdot \xi}
\end{equation}
with constants $c_j\in\CC^{m}$ and points $x_j\in B_R$, so that we have the bound
\begin{equation}\label{eqdisc}
\sup_{\xi\in \SS^{n-1}}\bigg|\int_{B_R}\hg(x)\, e^{-ix\cdot\xi}\, dx-f(\xi)\bigg|<\de'/2\,.
\end{equation}

Indeed, consider a covering of the ball $B_R$ by closed sets
$\{U_j\}_{j=1}^{N'}$, with piecewise smooth boundaries, pairwise
disjoint interiors, and diameters not exceeding $\de''$. Since the
function~$e^{-ix\cdot\xi}\, \hg(x)$ is smooth, we have that for each $x,y\in U_j$
\[
\sup_{\xi\in\SS^{n-1}}\big|\hg(x)\, e^{-ix\cdot\xi}-\hg(y)\, e^{-i y\cdot\xi}|< C\de''\,,
\]
with the constant $C$ depending on~$\hg$ (and therefore on~$\de'$) but
not on $\de''$. If $x_j$ is any point
in~$U_j$ and we set $c_j:=(2\pi)^{\frac{n}{2}}\hg(x_j)\,|U_j|$ in~\eqref{fcn}, we get
\begin{align*}
\sup_{\xi\in\SS^{n-1}}\bigg|\int_{B_R}\hg(x)\, e^{-ix\cdot\xi}\,
  dx-f(\xi)\bigg|&\leq \sum_{j=1}^{N'}\int_{U_j} \sup_{\xi\in\SS^{n-1}}\big|\hg(x)\,\e^{-i
                   x\cdot\xi}-\hg(x_j)\, e^{-ix_j\cdot\xi}\big|\, dx\\
&\leq C\de''\,,
\end{align*}
with $C$ depending on $\de'$ and $R$ but not on $\de''$ nor $N'$. By taking ~$\de''$ so that $C\de''<\de'/2$, the estimate~\eqref{eqdisc} follows.

Now, in view of ~\eqref{Rlarge} and~\eqref{eqdisc}, one has
\begin{equation*}
\|f-g\|_{C^0(\SS^{n-1})}<\de'\,,
\end{equation*}
so the estimate~\eqref{eqqma} follows upon noticing that the function~$f_2$ is the restriction
to~$\SS^{n-1}$ of the function~$g$.

To conclude, set
\begin{multline*}
\tvphi(x):=\int_{\SS^{n-1}}f(\xi)\, e^{ix\cdot\xi}\, d\si(\xi)=\sum_{j=1}^{N'}
\frac{c_j}{(2\pi)^{\frac{n}{2}}} \int_{\SS^{n-1}}e^{i(x-x_{j})\cdot \xi}\,d\si(\xi)=\\=\sum_{j=1}^{N'}
\frac{c_j}{|x-x_j|^{\frac{n}{2}-1}}J_{\frac{n}{2}-1}(|x-x_j|)\,,
\end{multline*}
then from Equation~\eqref{eqqma} we infer that
\begin{equation*}
\|\tvphi-\phi_2\|_{C^0(\RR^n)}\leq \int_{\SS^{n-1}}|f(\xi)-f_2(\xi)|\, d\si(\xi)<C\de'\,,
\end{equation*}
and from Equations~\eqref{L2b} and~\eqref{eqv2v1} we get the
$L^2$ estimate
\begin{multline}\label{cas}
\|\tphi-\tvphi\|_{L^2(B_2)}\leq C\|\tvphi-\phi_2\|_{C^0(\RR^n)}+C\|\phi_2-\phi_1\|_{C^0(\RR^n)}+\\+\|\phi_1-\tphi\|_{L^2(B_2)}<C\de'\,.
\end{multline}
Furthermore, both $\tvphi$ and $\tphi$ are $\CC^{m}$-valued functions satisfying the Helmholtz equation in $\RR^n$ (note that the Fourier transform of~$\tvphi$ is supported on~$\SS^{n-1}$), so by standard elliptic regularity estimates we have
\begin{equation*}
 \|\tphi-\tvphi\|_{C^{r}(B)}\leq C\|\tphi-\tvphi\|_{L^2(B_2)}< C\delta'\,.
\end{equation*}
This in particular implies that
\[
 \|\phi-\text{Re }\tvphi\|_{C^{r}(B)}< C\delta'\,,
\]
and taking $\delta'$ small enough so that $C\delta'<\de$, resetting $c_j:=\text{Re } c_j$, and defining $\vphi:=\text{Re }\tvphi$, the proposition follows. \end{proof}

The second step consists in showing that, for any large enough integer
$N$, we can find an $\RR^{m}$-valued eigenfunction $\psi$ of the Laplacian on $\SS^n$ with
eigenvalue $N(N+n-1)$ that approximates, in the ball $\BB_{1/N}$, when appropriately rescaled, the function $\vphi$ in the unit ball. The proof is based on
asymptotic expansions of ultraspherical polynomials, and uses the representation of~$\vphi$ as a sum of shifted Bessel functions which we obtained in the previous proposition as a key ingredient. It is then straightforward that Theorem~\ref{T.spharm1} follows from Propositions~\ref{P.Bessel1} and~\ref{P.spharm2}, provided that $N$ is large enough and $\de$ is chosen so that $2\de<\de'$.

\begin{proposition}\label{P.spharm2}
Given a constant $\de>0$, for any large enough positive integer $N$ there is an $\RR^{m}$-valued eigenfunction $\psi$ of the Laplacian on
$\SS^n$ with eigenvalue~$N(N+n-1)$ satisfying
\begin{equation*}
\bigg\|\vphi-\psi\circ \Psi^{-1}\Big(\frac\cdot k\Big)\bigg\|_{C^{r}(B)}<\de\,.
\end{equation*}
\end{proposition}

\begin{proof}
Consider the ultraspherical polynomial of dimension
$n+1$ and degree~$N$, $C^{n}_{N}(t)$, which is defined as
\begin{equation}\label{cla}
C^{n}_N(t):=\frac{\Gamma(N+1) \Gamma(\frac{n}{2})}{\Gamma(N+\frac{n}{2})}\,P_{N}^{(\frac{n}{2}-1,\,\frac{n}{2}-1)}(t)\,,
\end{equation}
where $\Gamma(t)$ is the gamma function and $P_{N}^{(\alpha,\,\beta)}(t)$ are the Jacobi polynomials (see e.g~\cite[Chapter IV, Section 4.7]{Szego75}). We have included a normalizing factor so that $C^{n}_N(1)=1$ for all $N$.

Let $p,q$ be two points in $\SS^{n}$, considered as vectors in $\RR^{n+1}$ with $|p|=|q|=1$. The addition theorem for ultraspherical
polynomials ensures that $C^{n}_N(p\cdot q)$ (where $p\cdot q$
denotes the scalar product in~$\RR^{n+1}$ of the vectors $p$ and $q$) can be written as
\begin{equation}\label{clasp}
C^{n}_{N}(p\cdot q)=\frac{2\pi^{\frac{n+1}{2}}}{\Gamma(\frac{n+1}{2})} \frac{1}{d(N, n)}\sum_{k=1}^{d(N, n)}Y_{N k}(p)\,Y_{N k}(q)\,,
\end{equation}
with $\{Y_{N k}\}_{k=1}^{d(N, n)}$ being an arbitrary orthonormal
basis of eigenfunctions of the Laplacian on $\SS^n$ (spherical harmonics) with eigenvalue $N(N+n-1)$.

The function $\vphi$ is written as the finite sum
\[
\vphi(x)=\sum_{j=1}^{N'} \frac{c_j}{|x-x_j|^{\frac{n}{2}-1}}J_{\frac{n}{2}-1}(|x-x_j|)\,,
\]
with coefficients $c_j\in \RR^{m}$ and points $x_j \in B_R$. With these $c_j$ and $x_j$ we define, for any point $p\in\SS^n$, the function
\begin{equation*}
\psi(p):=\sum_{j=1}^{N'} \frac{c_j}{2^{\frac{n}{2}-1}\Gamma(\frac{n}{2})} \, C^{n}_N(p\cdot p_j)\,,
\end{equation*}
where $p_j:=\Psi^{-1}(\frac{x_{j}}{N})$. As long as $N>R$, $p_j$ is well defined. In view of Equation~\eqref{clasp} it is clear that $\psi$ is an $\RR^m$-valued eigenfunction of the Laplacian on $\SS^n$ with eigenvalue $N(N+n-1)$.

Our aim is to study the asymptotic properties of the eigenfunction
$\psi$. To begin with, note that if we consider points $p:=\Psi^{-1}(\frac{x}{N})$ with $N>R$ and $x\in B_R$, we have
\begin{equation}\label{cos}
p\cdot p_j=\cos\big(\text{dist}_{\SS^n}(p,p_j)\big)=\cos \bigg(\frac{|x-x_j|+O(N^{-1})}{N}\bigg)\,,
\end{equation}
as $N\to \infty$. The last equality comes from $\Psi:\BB\to B$ being a patch of normal geodesic coordinates (by $\text{dist}_{\SS^n}(p,p_j)$ we mean the
distance between $p$ and $p_j$ considered on the sphere
$\SS^n$).  From now on we set
\begin{equation}\label{defi}
\tpsi(x):=\psi\circ\Psi^{-1}\bigg(\frac{x}{N}\bigg)\,.
\end{equation}

When $N$ is large, one has
$$
\frac{\Gamma(N+1)}{\Gamma(N+\frac{n}{2})}=N^{1-\frac{n}{2}}+O(N^{-\frac{n}{2}})\,,
$$
so from Equation~\eqref{cos} we infer
\begin{align}\label{aaa}
C^{n}_N(p\cdot
p_j)=\bigg(\Gamma \Big(\frac{n}{2}\Big)N^{1-\frac{n}{2}}+O(N^{-\frac{n}{2}})\bigg)\,
P_{N}^{(\frac{n}{2}-1,\,\frac{n}{2}-1)}\bigg(
\cos\bigg(\frac{|x-x_j|+O(N^{-1})}{N}\bigg)\bigg)\,.
\end{align}
By virtue of Darboux's formula for the Jacobi polynomials~\cite[Theorem 8.1.1]{Szego75}, we have the estimate
\begin{equation*}
\frac{1}{N^{\frac{n}{2}-1}}P_{N}^{(\frac{n}{2}-1,\,\frac{n}{2}-1)}\Big(\cos\frac{t}{N}\Big)=
2^{\frac{n}{2}-1}\, \frac{J_{\frac{n}{2}-1}(t)}{t^{\frac{n}{2}-1}}+O(N^{-1})\,,
\end{equation*}
uniformly in compact sets (e.g., for $|t|\leq
2R$). Hence, in view of Equation \eqref{aaa}, the function $\tpsi$ can be written as
\begin{align*}
\tpsi(x)&=\sum_{j=1}^{N'} \frac{c_j}{2^{\frac{n}{2}-1}\Gamma(\frac{n}{2})} C^{n}_N\Big(\cos\Big(\frac{|x-x_j|+O(N^{-1})}{N}\Big)\Big)\\&=\sum_{j=1}^{N'} \frac{c_j}{|x-x_j|^{\frac{n}{2}-1}}J_{\frac{n}{2}-1}(|x-x_j|)+O(N^{-1})\,,
\end{align*}
for $N$ big enough and $x, x_j\in B_R$. From this we get the
uniform bound
\begin{equation}\label{c0}
\|\vphi-\tpsi\|_{C^0(B)}<\de'\,
\end{equation}
for any $\delta'>0$ and all $N$ large enough.

It remains to promote this bound to a $C^{r}$ estimate. For this, note that, since the eigenfunction $\psi$ has eigenvalue $N(N+n-1)$, the rescaled function~$\tpsi$ verifies on $B$ the equation
\begin{equation*}
\Delta \tpsi+\tpsi=\frac{1}{N}A\tpsi\,,
\end{equation*}
with
$$
A\tpsi:=-(n-1)\tpsi+G_1\, \partial \tpsi+G_2\, \partial^2\tpsi\,,
$$
where $\partial^k \tpsi$ is a matrix whose entries are $k$-th order derivatives of $\tpsi$, and $G_k(x, N)$ are smooth matrix-valued functions with uniformly bounded derivatives, i.e.,
\begin{equation}\label{bg2a}
\sup_{x\in B}|\partial^\alpha_x G_k(x, N)|\leq C_{\al}\,,
\end{equation}
with constants $C_{\al}$ independent of $N$.

Since $\vphi$ satisfies the Helmholtz equation $\De \vphi+\vphi=0\,$, the difference $\vphi-\tpsi$ satisfies
$$
\De(\vphi-\tpsi)+(\vphi-\tpsi)=\frac{1}{N}A\tpsi\,,
$$
and, considering the estimates~\eqref{c0} and~\eqref{bg2a}, by standard elliptic estimates we get
\begin{align*}
\|\vphi-\tpsi\|_{C^{r,\alpha}(B)}&<C\|\vphi-\tpsi\|_{C^0(B)}+\frac{C}{N}\|A\tpsi\|_{C^{r-2,\alpha}(B)}\\
&<C\delta'+\frac{C}{N}\|\vphi-\tpsi\|_{C^{r,\al}(B)}+\frac{C}{N}\|\vphi\|_{C^{r,\alpha}(B)}\,,
\end{align*}
so we conclude that, for $N$ big enough and $\de'$
small enough,
$$
\|\vphi-\tpsi\|_{C^{r}(B)}\leq C\de'+\frac{C \|\vphi\|_{C^{r,\al}}}{N}<\de\,.
$$
The proposition then follows. \end{proof}

\section{An inverse localization theorem on the torus}\label{IL_torus}

In this section we prove an inverse localization theorem for eigenfunctions of the Laplacian on $\TT^n$ for $n\geq 3$. We recall that the eigenvalues of the Laplacian on the $n$-torus are the integers of the form
$$
\la=|k|^2
$$
for some $k\in\ZZ^n$. In particular, the spectrum of the Laplacian in $\TT^n$ contains the set of the squares of integers.

As in the previous section, we fix an arbitrary point $p_0\in \TT^n$ and take a patch of normal geodesic coordinates
$\Psi:\BB\to B$ centered at $p_0$.

\begin{theorem}\label{T.torus}
Let $\phi$ be an $\RR^{m}$-valued function in $\RR^n$, satisfying
$\Delta \phi+\phi=0$. Fix a positive integer $r$ and a positive constant $\de'$. For any large enough odd integer $N$, there is an $\RR^m$-valued eigenfunction $\psi$ of the Laplacian on $\TT^n$ with eigenvalue $N^2$ such that
\begin{equation*}
\bigg\|\phi-\psi\circ \Psi^{-1}\Big(\frac\cdot N\Big)\bigg\|_{C^{r}(B)}\leq\delta'\,.
\end{equation*}
\end{theorem}
\begin{proof}
Arguing as in the proof of Proposition~\ref{P.Bessel1} we can readily
show that for any $\de>0$, there exists an $\RR^{m}$-valued function $\phi_1$ on $\RR^n$ that approximates the function $\phi$ in the ball $B$ as
\begin{equation}\label{otrv}
\|\phi_1-\phi\|_{C^{0}(B)}<\de\,,
\end{equation}
and that can be represented as the Fourier transform of a distribution supported on the unit sphere of the form
\begin{equation*}
\phi_1(x)=\int_{\SS^{n-1}}f(\xi)\, e^{i\xi\cdot x}\, d\si(\xi)\,.
\end{equation*}
Again $\SS^{n-1}$ denotes the unit sphere $\{\xi\in\RR^n:|\xi|=1\}$ and
$f$ is a smooth $\CC^m$-valued function on $\SS^{n-1}$ satisfying $f(\xi)=\bar{f}(-\xi)$.

Let us now cover the sphere $\SS^{n-1}$ by finitely many closed sets
$\{U_k\}_{k=1}^{N'}$ with piecewise smooth boundaries and pairwise
disjoint interiors such that the
diameter of each set is at most $\epsilon$. We can then repeat the argument
used in the proof of Proposition~\ref{P.Bessel1} to infer that, if
$\xi_k$ is any point in $U_k$ and we set
\[
c_k:=f(\xi_k)\, |U_k|\,,
\]
the function
\[
\tpsi(x):=\sum_{k=1}^{N'} c_{k}\, e^{i\xi_{k}\cdot x}
\]
approximates the function $\phi_1$ uniformly with an error proportional to~$\epsilon$:
\begin{equation*}%\label{mases}
\|\tpsi-\phi_1\|_{C^{0}(B)}<C\epsilon\,.
\end{equation*}
The constant $C$ depends on $\de$ but not on $\epsilon$ nor $N'$, so one can choose
the maximal diameter~$\epsilon$ small enough so that
\begin{equation}\label{mases}
\|\tpsi-\phi_1\|_{C^{0}(B)}<\delta\,.
\end{equation}
In turn, the uniform estimate
\[
\|\tpsi-\phi\|_{C^0(B)}\leq \|\tpsi-\phi_1\|_{C^0(B)}+ \|\phi-\phi_1\|_{C^0(B)}<2\de
\]
can be readily promoted to the $C^{r}$ bound
\begin{equation}\label{estder}
\|\tpsi-\phi\|_{C^{r}(B)}<C\de\,.
\end{equation}
This follows from standard elliptic estimates as both $\tpsi$ (whose
Fourier transform is supported on~$\SS^{n-1}$) and $\phi$ satisfy the
Helmholtz equation:
\[
\De \tpsi +\tpsi =0\,,\qquad \De \phi+ \phi=0\,.
\]

Furthermore, replacing $\tpsi$ by its real part if necessary, we can safely
assume that the function $\tpsi$ is $\RR^m$-valued.

Let us now observe that for any large enough odd integer~$N$ one can choose the points $\xi_k\in U_k\subset\SS^{n-1}$ so
that they have rational components (i.e., $\xi_{k}\in
\mathbb{Q}^{n}$) and the rescalings $N\xi_k$ are integer
vectors (i.e., $N \xi_k\in\ZZ^n$). This is because for $n\geq3$, rational points
$\xi\in\SS^{n-1}\cap \mathbb{Q}^n$ of height $N$ (and so with $N\xi\in\ZZ^n$) are uniformly distributed on the unit sphere as $N\to \infty$ through odd values~\cite{Du03} (in fact, the requirement for $N$ to be odd can be dropped for $n\geq 4$~\cite{Du03}).

Choosing $\xi_k$ as above, we are now ready to prove
the inverse localization theorem in the torus. Without loss of generality, we
will take the origin as the base point $p_0$, so that we can identify
the ball $\BB$ with $B$ through the canonical $2\pi$-periodic
coordinates on the torus. In particular, the
diffeomorphism~$\Psi:\BB\to B$
that appears in the statement of the theorem can be
understood to be the identity.

Since $N
\xi_k\in\ZZ^n$, it follows that the function
\begin{equation*}
\psi(x):=\sum_{k=1}^{N'}c_{k}e^{iN\xi_k\cdot x}
\end{equation*}
is $2\pi$-periodic (that is, invariant under the translation $x\to
x+2\pi\, a$ for any vector $a\in\ZZ^n$). Therefore it defines a well-defined function on the torus, which we will still denote by $\psi$.

Since the Fourier transform of $\psi$ is now supported on the sphere of
radius~$N$, $\psi$ is an eigenfunction of the Laplacian on the torus $\TT^n$ with eigenvalue $N^2$,
$$
\Delta \psi+N^{2}\psi=0\,.
$$
The theorem then follows provided that $\de$ is chosen
small enough for $C\de<\de'$.
\end{proof}

We conclude this section noticing that the statement of Theorem~\ref{T.torus} does not hold for $\TT^2$. The reason is that rational points
$\xi\in\SS^{1}\cap \mathbb{Q}^2$ with $N\xi\in\ZZ^2$ are no longer uniformly distributed on the unit circle (not even dense) as $N\to \infty$ through any sequence of odd values, counterexamples can be found in~\cite{Cille}. Nevertheless, a slightly different statement can be proved using~\cite{Cille}:
\begin{theorem}
Let $\phi$ be an $\RR^{m}$-valued function in $\RR^2$, satisfying
$\Delta \phi+\phi=0$. Fix a positive integer $r$ and a positive constant $\de'$. Then there exists a sequence of integers $\{N_l\}_{l=1}^\infty\nearrow  \infty$, and $\RR^m$-valued eigenfunctions $\psi_l$ of the Laplacian on $\TT^2$ with eigenvalues $N_l^2$ such that
\begin{equation*}
\bigg\|\phi-\psi_l\circ \Psi^{-1}\Big(\frac\cdot N_l\Big)\bigg\|_{C^{r}(B)}\leq\delta'
\end{equation*}
for $l$ large enough.
\end{theorem}

\section{Proof of the main theorem}\label{P.main}

For the ease of notation, we shall write $\MM$ to denote
either~$\TT^n$ or~$\SS^n$. Let $\Phi'$ be a diffeomorphism of $\MM$ mapping the codimension $m$ submanifold $\Sigma$ into the ball $\BB_{1/2} \subset \MM$, and the ball $\BB_{1/2}$ into itself. In $\SS^n$, the existence of such a diffeomorphism is
trivial, while in the case of $\TT^n$ it follows from the assumption
that $\Sigma$ is contained in a contractible set.

Consider the submanifold $\Sigma'$
in $B_{1/2} \subset \RR^n$ defined as $\Phi'(\Sigma)$ in the path of normal geodesic coordinates:
$$
\Sigma':=(\Psi \circ \Phi')(\Sigma)\,.
$$
It is shown in~\cite[Theorem 1.3]{EP} if $m\geq 2$ and~\cite[Remark A.2]{EP} if $m=1$, that there is an $\RR^m$-valued monochromatic wave $\phi=(\phi_1,\cdots,\phi_m)$, satisfying $\Delta \phi+\phi=0$ in $\RR^n$, and a diffeomorphism $\Phi_1$ (close to the identity, and different from the identity only on $B_{1/2}$) such that $\Phi_1(\Sigma')\subset B_{1/2}$ is a union of connected components of the joint nodal set $\phi_{1}^{-1}(0) \cap...\cap \phi_{m}^{-1}(0)$. In addition, the construction in~\cite{EP} ensures that the regularity condition $rk(\nabla \phi_1,\cdots,\nabla \phi_m)=m$ holds at any point of $\Phi_1(\Sigma')$, so it is a structurally stable nodal set of $\phi$ by Thom's isotopy theorem~\cite{AR}.

Now, the inverse localization theorem (Theorem \ref{T.spharm1} in the case of $\SS^n$ and Theorem \ref{T.torus} for $\TT^n$) allows us to find, for any large enough odd integer $N$, an $\RR^m$-valued function $\psi=(\psi_1,\cdots,\psi_m)$ in $\MM$ satisfying $\Delta \psi=- \lambda \psi$ (with $\lambda:=N(N+n-2)$ or $\lambda:=N^2$ in the sphere or the torus, respectively) and such that $\psi \circ \Psi^{-1}(\frac{\cdot} {N})$ approximates $\phi$ in the $C^{r}(B)$ norm as much as we want.

The structural stability property ensures the existence of a second diffeomorphism $\Phi_2: \RR^{n} \rightarrow \RR^{n}$ close to the identity, and different from the identity only on $B_{1/2}$, such that $\Phi_2(\Phi_1(\Sigma'))$ is a union of connected components of the joint nodal set of the $\RR^m$-valued function $\psi \circ \Psi^{-1}(\frac{\cdot} {N})$. Therefore, the corresponding submanifold
$$\Phi(\Sigma):=\Psi^{-1}\Big(\frac{1}{N}\Phi_2(\Phi_1((\Psi \circ \Phi')(\Sigma)))\Big)$$
is a union of connected components of the nodal set of $\psi$. The map $\Phi: \MM \rightarrow \BB_{\frac{1}{2N}}$ thus defined is easily extended to a diffeomorphism of the whole manifold $\MM$. Finally, we have by the construction that $\Phi(\Sigma)$ is structurally stable, and hence Theorem~\ref{T.high} follows.

\section{Final remark: inverse localization on the sphere in multiple regions}\label{multiple}

Theorem~\ref{T.spharm1} in Section~\ref{IL_sphere} can be refined to include inverse localization at
different points of the sphere. This way, we get an eigenfunction of the Laplacian that approximates several given solutions of the Helmholtz equation in different regions. The fast decay of ultraspherical polynomials of high degree outside the domains where they behave as shifted Bessel functions is behind this multiple localization.  Notice that, in contrast, trigonometric polynomials do not exhibit this decay, hence the lack of an analog of the following result in the case of the torus. All along this section we assume that $n\geq 2$.

Let $\{p_\alpha \}_{\alpha=1}^{N'}$ be a set of points in $\SS^{n}$, with $N'$ an arbitrarily large (but fixed throughout) integer. We denote by $\Psi_{\alpha} : \BB_{\rho}(p_\alpha) \rightarrow B_\rho$ the corresponding geodesic patches on balls of radius $\rho$ centered at the points $p_{\alpha}$. We fix a radius $\rho$ such that no two balls intersect, for example by setting
\[
\rho:=\frac12\min_{\alpha \neq \beta}\dist_{\SS^n}(p_\alpha, p_\beta)\,.
\]
We further choose the points $\{p_\alpha\}_{\alpha=1}^{N'}$ so that no pair of points are antipodal in $\SS^{n}\subset \RR^{n+1}$, i.e. $p_\alpha\neq-p_\beta$ for all $\alpha$, $\beta$. The reason is that the eigenfunctions of the Laplacian on the sphere with eigenvalue $N(N+n-1)$ have parity $(-1)^{N}$:
\[
\psi(p_{\alpha})=(-1)^{N} \psi(-p_{\alpha})
\]
(they are the restriction to the sphere of homogenous harmonic polynomials of degree $N$); so that prescribing the behavior of an eigenfunction in a ball around the point $p_{\alpha}$ automatically determines its behavior in the antipodal ball.

\begin{proposition}\label{P.spharm.corr}
Let $\{\phi_\alpha\}_{\alpha=1}^{N'}$ be a set of $N'$ $\RR^{m}$-valued monochromatic waves in $\RR^n$, $1\leq m\leq n$, satisfying
$\Delta \phi_\alpha+\phi_\alpha=0$. Fix a positive integer $r$ and a positive constant $\de$.  For any large enough integer $N$, there is an $\RR^{m}$-valued eigenfunction $\psi$ of the Laplacian on $\SS^n$ with eigenvalue $N(N+n-1)$ such that
\begin{equation*}%\label{prop2e1}
\bigg\|\phi_\alpha-\psi\circ \Psi_{\alpha}^{-1}\Big(\frac\cdot N\Big)\bigg\|_{C^{r}(B)}<\de\,
\end{equation*}
for all $1 \leq \alpha \leq N'$.
\end{proposition}
\begin{proof} We use the notation introduced in the proof of Proposition~\ref{P.spharm2} without further mention. Applying Theorem~\ref{T.spharm1} to each $\phi_\alpha$ we obtain, for high enough $N$, $\RR^{m}$-valued eigenfunctions of the Laplacian $\{ \psi_\alpha \}_{\alpha=1}^{N'}$ satisfying the bound
\begin{equation*}
\bigg\|\phi_\alpha-\psi_\alpha\circ \Psi_{\alpha}^{-1}\Big(\frac\cdot N\Big)\bigg\|_{C^{r}(B)}<\de'\,.
\end{equation*}
For each $\alpha$, the $\RR^{m}$-valued eigenfunction $\psi_\alpha(p)$ is a linear combination (with coefficients in $\RR^{m}$) of ultraspherical polynomials  $C^{n}_{N}(p\cdot q_j)$, where $\{q_j\}$ is a finite set of points such that $\dist_{\SS^n}(p_\alpha, q_j)$ is proportional to $N^{-1}$, for all $j$. Recall that the ultraspherical polynomials satisfy the asymptotic formula
\[
C^{n}_N(p\cdot
q)=\frac{\Gamma(\frac{n}{2})}{N^{\frac{n}{2}-1}}\,
P_{N}^{(\frac{n}{2}-1,\,\frac{n}{2}-1)}(
\cos(\dist_{\SS^n}(p, q)))+O(N^{-\frac{n}{2}})\,,
\]
so considering the fact that the Jacobi polynomials behave as (see ~\cite[Theorem 7.32.2]{Szego75})
\[
N^{1-\frac{n}{2}}\, P_N^{(\frac{n}{2}-1,\frac{n}{2}-1)}(\cos t)=\frac{O(N^{-1})}t,
\]
uniformly for $N^{-1}<t<\pi-N^{-1}$, we can conclude that the functions $C^{n}_N(p \cdot q_j)$ are uniformly
bounded as
\[
|C^{n}_N(p\cdot q_j)|\leq \frac{C_\rho}{N}
\]
for any point $p$ satisfying
\begin{equation*}
\text{min}_j \dist_{\SS^n}(p,q_j)\geq \rho \quad \text{and}\quad \text{min}_j \dist_{\SS^n}(p,-q_j)\geq \rho \,,
\end{equation*}
and where $C_{\rho}$ is a constant depending only on $\rho$.
The same decay is thus also exhibited by the eigenfunction $\psi_\alpha$,
\[
\|\psi_{\alpha}\|_{C^0(\SS^n\backslash (\BB(p_\alpha,\rho)\cup \BB(-p_\alpha,\rho))}\leq \frac{C}{N}
\]
since it is just a normalized linear combination of ultraspherical polynomials (here the constant $C$ depends on $\rho $ and on the particular coefficients in the expansion of $\psi_\alpha$, that is, on $\phi_\alpha$ and $\delta'$).

Now, if we define the $\RR^{m}$-valued eigenfunction
\[
\psi:=\sum_{\alpha=1}^{N'} \psi_{\alpha}\,
\]
and we choose $N$ large enough, the statement of the proposition follows for $r=0$. By standard elliptic estimates, the $C^0$ bound can be easily promoted to a $C^{r}$ bound, so we are done. \end{proof}

\bibliographystyle{amsplain}

\begin{thebibliography}{99}\frenchspacing

\bibitem{AR}
R. Abraham and J. Robbin, \emph{Transversal mappings and flows}, Benjamin, New York, 1967.

%\bibitem{BerryDennis}
%
%M. Berry, M. Dennis, Knotted and linked phase singularities in monochromatic waves.  R. Soc. Lond. Proc. A 457 (2001) 2251--2263.

\bibitem{Berry}

M. Berry, Knotted zeros in the quantum states of hydrogen. Found. Phys. 31 (2001) 659--667.

%\bibitem{Du88}
%W. Duke, Hyperbolic distribution problems and half-integral weight Maass forms, Invent. Math. 92 (1988) 73--90.

\bibitem{Cille}
J. Cilleruelo, The distribution of the lattice points on circles. J. Number Theor. 43 (1993) 198--202.

\bibitem{Du03}
W. Duke, Rational points on the sphere, Ramanujan J. 7 (2003) 235--239.

\bibitem{EHP16}
A. Enciso, D. Hartley, D. Peralta-Salas, Laplace operators with eigenfunctions whose nodal set is a knot. J. Funct. Anal. 271 (2016) 182--200.

\bibitem{EPH1}
A. Enciso, D. Hartley, D. Peralta-Salas, A problem of Berry and knotted zeros in the eigenfunctions of the harmonic oscillator. J. Eur. Math. Soc. 20 (2018) 301--314.

\bibitem{EPH2}
A. Enciso, D. Hartley, D. Peralta-Salas, Dislocations of arbitrary topology in Coulomb eigenfunctions. Rev. Mat. Iberoamericana 34 (2018) 1361--1371.

\bibitem{EP}
A. Enciso, D. Peralta-Salas, Submanifolds that are level sets of solutions to a second-order elliptic PDE, Adv. Math. 249 (2013) 204--249.

\bibitem{EP16}
A. Enciso, D. Peralta-Salas, Eigenfunctions with prescribed nodal sets. J. Differential Geom. 101 (2015) 197--211.

\bibitem{EPS17}
A. Enciso, D. Peralta-Salas, S. Steinerberger, Prescribing the nodal set of the first eigenfunction in each
conformal class. Int. Math. Res. Not. 11 (2017) 3322--3349.

\bibitem{EPT}
A. Enciso, D. Peralta-Salas, F. Torres de Lizaur, Knotted structures in high-energy Beltrami fields on the torus and the sphere, Ann. Sci. \'Ec. Norm. Sup. 50 (2017) 4, 995--1016.

\bibitem{Hormander}
L. H\"ormander, \emph{The analysis of linear partial differential
  operators I}, Springer, Berlin, 2003.

\bibitem{Lo1}
A. Logunov, Nodal sets of Laplace eigenfunctions: polynomial upper estimates of the Hausdorff measure. Ann. of Math. 187 (2018) 221--239.

\bibitem{Lo2}
A. Logunov, Nodal sets of Laplace eigenfunctions: proof of Nadirashvili's conjecture
and of the lower bound in Yau's conjecture.  Ann. of Math. 187 (2018) 241--262.

\bibitem{Ma59}
W. S. Massey, On the normal bundle of a sphere imbedded in Euclidean space, Proc. Amer. Math. Soc. 10 (1959) 959--964.

\bibitem{Szego75}
G. Szeg{\H{o}}, {\em Orthogonal polynomials}, AMS, Providence, 1975.

\bibitem{T18}
F. Torres de Lizaur, Geometric structures in the nodal sets of eigenfunctions of the Dirac operator. Preprint, arXiv:1712.10310.

\bibitem{Ya82}
S.T. Yau, Problem section, Seminar on Differential Geometry, Annals of Mathematics Studies 102 (1982) 669--706.

\bibitem{Ya93}
S.T. Yau, Open problems in geometry, Proc. Sympos. Pure Math. 54, pp. 1--28, Amer. Math. Soc., Providence, 1993.

\end{thebibliography}

\end{document}